\documentclass{amsart}

\usepackage{graphicx}
\usepackage[margin=3cm]{geometry}

\newtheorem{theorem}{Theorem}[section]
\newtheorem{lemma}[theorem]{Lemma}

\theoremstyle{definition}

\theoremstyle{remark}
\newtheorem{remark}[theorem]{Remark}

\numberwithin{equation}{section}

\begin{document}

\title{Non-local logistic equations from the probability viewpoint}

\author{Mirko D'Ovidio}
\address{Department of Basic and applied Sciences for Engineering,  Sapienza University of Rome, Italy }
\email{mirko.dovidio@uniroma1.it}
\thanks{The author was supported in part by INDAM-GNAMPA and the Grant Ateneo "Sapienza 2019".}



\begin{abstract}
We investigate the solution to the logistic equation involving non-local operators in time. In the linear case such operators lead to the well-known theory of time changes. We provide the probabilistic representation for the non-linear logistic equation with non-local operators in time. The so-called fractional logistic equation has been investigated by many researchers, the problem to find the explicit representation of the solution on the whole real line is still open. In our recent work the solution on compact sets has been written in terms of Euler's numbers.
\end{abstract}

\maketitle

\section{Introduction}

The study of the logistic and fractional logistic growth has attracted many researchers because of the high impact in the applied sciences. Here we bring to the reader's attention some of the recently appeared papers, \cite{AN2021,DCDC2021,IzaSri2020,KahPhaJam2020,KumRaj2020} in which the role of the logistic equation has been investigated and discussed. Concerning the the stochastic interpretation and representation, a comparison with other models of growth has been given recently in \cite{DCP}. Many other works are devoted to the logistic SDEs. The literature is huge, we mention only few works here and further on in the presentation of the results.  

In this short note, we discuss some aspects concerning logistic equations and random time changes. The fractional logistic equation has been investigated by many researchers in the last decade. However, the solutions has been obtained only recently in \cite{PhysA}. Many similar problems have been considered in order to find solutions sharing some peculiar properties with the solution to the fractional logistic equation. A discussion on this point has been given in \cite{FCAAlog} and the references therein.  

We consider non-local operators more general than the Caputo-Djrbashian fractional derivative. Then, we introduce non-local logistic equations and discuss the probabilistic representation of the solutions in terms of inverses to subordinators. Thus, we use stochastic processes driven by non-local partial differential equations in order to solve non-local logistic equations.

Our presentation is based on two cases concerning respectively the logistic equation on the real line and the logistic equation on compact sets.

\section{Preliminaries on fractional Calculus and probability}
Let $H=\{H_t\, , \, t \geq 0\}$ be a subordinator (see \cite{BerBook} for a detailed discussion). Then, $H$ can be characterized by the Laplace exponent $\Phi$, that is, $$\mathbf{E}_0[\exp( - \lambda H_t)] = \exp(- t \Phi(\lambda)),$$
$\lambda \geq 0$. As usual we denote by $\mathbf{E}_x$ the expected value w.r. to $\mathbf{P}_x$ where $x$ is a starting point. Moreover, if $\Phi$ is the Laplace exponent of a subordinator, then there exists a unique pair $(\mathtt{k}, \mathtt{d})$ of non-negative real numbers and a unique measure $\Pi$ on $(0, \infty)$ with $\int (1 \wedge z) \Pi(dz) < \infty$, such that for every $\lambda \geq 0$
\begin{equation}
\Phi(\lambda) = \mathtt{k} + \mathtt{d} \lambda + \int_0^\infty \left( 1 - e^{ - \lambda z} \right) \Pi(dz). \label{LevKinFormula}
\end{equation} 
The L\'evy-Khintchine representation in formula \eqref{LevKinFormula} is written in terms of the killing rate 
\begin{align*}
\mathtt{k}=\Phi(0)
\end{align*}
and the drift coefficient 
\begin{align*}
\mathtt{d}= \lim_{\lambda\to \infty} \frac{\Phi(\lambda)}{\lambda}
\end{align*}
where
\begin{align}
\label{tailSymb}
\frac{\Phi(\lambda)}{\lambda} = \mathtt{d} +  \int_0^\infty e^{-\lambda z} \overline{\Pi}(z)dz, \qquad \overline{\Pi}(z) = \mathtt{k} + \Pi((z, \infty))
\end{align}
and $\overline{\Pi}$ is the so called \emph{tail of the L\'evy measure}. We recall that $\Phi$ is uniquely given by \eqref{LevKinFormula}. In particular, it is a Bernstein function,  then $\Phi$ is non-negative, non-decreasing and continuous. For details, we refer to the well-known book \cite{BerBook}. The interested reader can also consult the recent book \cite{SchilingBook}. 

We define the inverse process $L =\{L_t\, , \, t \geq 0\}$ to a subordinator as $$L_t := \inf\{ s \geq 0\,:\, H_s \notin (0, t)\}.$$
We recall that $H_0=0$ and $L_0=0$. We do not consider step-processes with $\Pi((0,\infty))<\infty$,  we focus only on strictly increasing subordinators with infinite measures (then $L$ turns out to be a continuous process). By definition of inverse process, we can write
\begin{align}
\label{relationPHL}
\mathbf{P}_0(L_t < s) = \mathbf{P}_0(H_s>t).
\end{align} 
We also denote by
\begin{align}
\label{densitiesHL}
h(t, x)dx=P(H_t \in dx) \quad \textrm{and} \quad l(t, x)dx=P(L_t \in dx)
\end{align}
the corresponding densities (see for example \cite{MS2008}). 
Further on we use the following potentials 
\begin{align}
\label{potentialH}
\int_0^\infty e^{-\xi x}\,h(t, x)\, dx = e^{- t \, \Phi(\xi)}, \quad \xi >0
\end{align}
and, after easy calculations involving \eqref{relationPHL} and \eqref{potentialH} (as already proved in \cite{MS2008})
\begin{align}
\label{potentialL}
\int_0^\infty e^{-\lambda t}\, l(t,x)\, dt = \frac{\Phi(\lambda)}{\lambda} e^{-x \Phi(\lambda)}, \quad \lambda>0.
\end{align}

Let $M>0$ and $w \geq 0$. Let $\mathcal{M}_w$ be the set of (piecewise) continuous function on $[0, \infty)$ of exponential order $w$ such that $|\varrho(t)| \leq M e^{wt}$. Denote by $\widetilde{\varrho}$ the Laplace transform of $\varrho$. Then, we define the operator $\mathfrak{D}^\Phi_t : \mathcal{M}_w \mapsto \mathcal{M}_w$ such that
\begin{align*}
\int_0^\infty e^{-\lambda t} \mathfrak{D}^\Phi_t \varrho(t)\, dt = \Phi(\lambda) \widetilde{\varrho}(\lambda) - \frac{\Phi(\lambda)}{\lambda} \varrho(0), \quad \lambda > w
\end{align*}
where $\Phi$ is given in \eqref{LevKinFormula}. Since $\varrho$ is exponentially bounded, the integral $\widetilde{\varrho}$ is absolutely convergent for $\lambda>w$.  The inverse Laplace transforms $\varrho$ and $\mathfrak{D}^\Phi_t \varrho$ are uniquely defined. Since
\begin{align}
\label{PhiConv}
\Phi(\lambda) \widetilde{\varrho}(\lambda) - \frac{\Phi(\lambda)}{\lambda} \varrho(0) = & \left( \lambda \widetilde{\varrho}(\lambda) - \varrho(0) \right) \frac{\Phi(\lambda)}{\lambda},
\end{align}
the function $\mathfrak{D}^\Phi_t \varrho$ can be written as a convolution involving the ordinary derivative and the inverse transform of \eqref{tailSymb} iff $\varrho \in \mathcal{M}_w \cap C([0, \infty), \mathbb{R}_+)$ and $\varrho^\prime \in \mathcal{M}_w$. We also observe that (Young's inequality)
\begin{align}
\label{YoungSymb}
\int_0^\infty |\mathfrak{D}^\Phi_t \varrho |^p dt \leq \left( \int_0^\infty |\varrho^\prime |^p dt \right) \left( \lim_{\lambda \downarrow 0} \frac{\Phi(\lambda)}{\lambda} \right)^p, \qquad p \in [1, \infty)
\end{align}
where $\lim_{\lambda \downarrow 0} \Phi(\lambda) /\lambda$ is finite only in some cases. For example, for $\mathtt{d}=0$ and $\mathtt{k}=0$:
\begin{itemize}
\item[i)] inverse Gaussian subordinator with $\Phi(\lambda) = \sigma^{-2} \left(\sqrt{2\lambda \sigma^2 +\mu^2} - \mu \right)$ with $\sigma \neq 0$;
\item[ii)] gamma subordinator with $\Phi(\lambda)= a \ln (1+ \lambda/b)$  with $ab>0$;
\item[iii)] generalized stable subordinator with $\Phi(\lambda) = (\lambda + \gamma)^\alpha - \gamma^\alpha$ with $\gamma>0$ and $\alpha \in (0,1)$.
\end{itemize}
Thus, for example, $\mathfrak{D}^\Phi_t \varrho \in L^1(0, \infty)$  if $\varrho^\prime \in L^1(0, \infty)$ and $\Phi^\prime(0)< \infty$. The operator $\mathfrak{D}^\Phi_t$, in alternative and sometimes slightly different forms, it has been first considered in \cite{SC2003} after in \cite{Koc2011} and recently in \cite{Chen2017,Toaldo2015}.

We introduce the following notation
\begin{align*}
\mathcal{M}^\prime_w =\{ \varphi \in C([0, \infty), \mathbb{R}_+ )\,:\, \varphi, \varphi^\prime \in \mathcal{M}_w \}.
\end{align*}

\begin{remark}
Let us recall a couple of special cases.
\begin{itemize}
\item[i)] We notice that when $\Phi(\lambda)=\lambda$ we have that $H_t = t$ and $L_t=t$ a.s. and in \eqref{YoungSymb} the equality holds. The operator $\mathfrak{D}^\Phi_t$ becomes the ordinary derivative
\begin{align*}
D_t \varrho (t)= \frac{d \varrho}{dt}(t).
\end{align*}  
\item[ii)] The well-known case $\Phi(\lambda)=\lambda^\alpha$, $\alpha\in (0,1)$ gives the Caputo-Djrbashian derivative
\begin{align}
D^\alpha_t \varrho(t) = \frac{1}{\Gamma(1-\alpha)} \int_0^t \varrho^\prime(s) \, (t-s)^{-\alpha}\, ds.
\end{align}
The corresponding processes are $H_t$ which is a stable subordinator and $L_t$ which is an inverse to a stable subordinator (\cite{MS2008, DOVspl2011, MS2013}).
\end{itemize}
\end{remark}

\begin{remark}
We recall the following result which will be useful below. Let us introduce the Riemann-Liouville (type) derivative 
\begin{align*}
\mathcal{D}^\Phi_t \varrho(t) = \frac{d}{dt} \int_0^t \varrho(s)\, \overline{\Pi}(t-s)\, ds.
\end{align*}
Between $\mathfrak{D}^\Phi_t$ and $\mathcal{D}^\Phi_t$ there exists the following relation
\begin{align}
\label{connection}
\mathfrak{D}^\Phi_t u(t) = \mathcal{D}^\Phi_t \big( u(t) - u(0) \big) = \mathcal{D}^\Phi_t u(t)  - u(0)\, \Pi((t, \infty))
\end{align}
where in the last step we have used the fact that
\begin{align*}
\mathcal{D}^\Phi_t \, c = c\, \Pi((t, \infty)) \quad \textrm{for a constant } c.
\end{align*}
The density $l$ of the process $L_t$ solves the following problem
\begin{equation}
\label{eqL}
\begin{array}{ll}
\displaystyle \mathcal{D}^\Phi_t l(t,x) = - \frac{\partial l}{\partial x}(t,x), \quad x>0,\, t>0,\\
\displaystyle l(t,0) = \Pi((t, \infty)),\\
\displaystyle l(0, x) = \delta(x).
\end{array}
\end{equation}
From the Laplace technique, by considering \eqref{PhiConv} and the potential \eqref{potentialL} we get immediately the result. We skip the proof (the reader can consult \cite{Toaldo2015}). 
\end{remark}

\section{Non-local linear equations}

\begin{lemma}
\label{thm:v}
Let $v \in \mathcal{M}^\prime_w$ with $w\geq 0$ be the solution to
\begin{align}
\label{eqv1}
\frac{dv}{dt} = f(v), \quad v(0)= c \geq 0
\end{align}
where $f$ is linear on $\mathcal{M}^\prime_w$. Then, $v$ is the unique classical solution to
\begin{align}
\label{eqDv1}
\mathfrak{D}^\Phi_t v = f(v) * \overline{\Pi}, \quad v(0) = c.  
\end{align}
\end{lemma}
\begin{proof}
Notice that $f$ on $\mathcal{M}^\prime_w$ together with \eqref{eqv1} say that $f: \mathcal{M}^\prime_w \to \mathcal{M}^\prime_w$. The Laplace transform of equation \eqref{eqv1} writes
\begin{align*}
\lambda \widetilde{v}(\lambda) - v(0) = f(\widetilde{v}(\lambda))
\end{align*}
or equivalently
\begin{align*}
\Phi(\lambda) \widetilde{v}(\lambda) - \frac{\Phi(\lambda)}{\lambda} v(0) = f\left( \frac{\Phi(\lambda)}{\lambda} \widetilde{v}(\lambda)\right).
\end{align*}
The latter has the following reading
\begin{align*}
\Phi(\lambda) \widetilde{v}(\lambda) - \frac{\Phi(\lambda)}{\lambda} v(0) = \int_0^\infty e^{-\lambda t} \mathfrak{D}^\Phi_t v(t) dt
\end{align*}
and
\begin{align*}
f\left( \frac{\Phi(\lambda)}{\lambda} \widetilde{v}(\lambda)\right)= \frac{\Phi(\lambda)}{\lambda}  f\left( \widetilde{v}(\lambda)\right) 
= & \frac{\Phi(\lambda)}{\lambda} \int_0^\infty f(v(z)) e^{-z \lambda} dz\\
= & \int_0^\infty e^{-\lambda t} \left( \int_0^t f(v(t-z)) \overline{\Pi}(z)dz \right) dt. 
\end{align*}
That is, the following pointwise equality holds, $\forall t >0$,
\begin{align*}
\mathfrak{D}^\Phi_t v(t) = \int_0^t f(v(t-z)) \overline{\Pi}(z)dz.
\end{align*}
\end{proof}

%
%

The solution to $\mathfrak{D}^\Phi_t u = f(u)$ where $f$ is linear can be written in terms of the density $l(t, x)$. This is expected in case of linear $f$ and it is well-known in case of linear operator $Av$. Indeed, the latter can be included in the theory of time-changed processes first introduced in \cite{BM2001} for the fractional (Caputo-Djrbashian) derivative and after, in \cite{Toaldo2015, Chen2017} for a general non-local operator. For the sake of completeness we provide the following statement.

\begin{theorem}
\label{thm:linf}
Let the setting of Lemma \ref{thm:v} prevails with $w\geq 0$ and $v \in \mathcal{M}^\prime_w$. Then, the function
\begin{align*}
\mathcal{M}^\prime_{w} \ni u(t) = \int_0^\infty v(x)l(t, dx)  =: \mathbf{E}_0[v(L_t)]
\end{align*}
is the unique classical solution to the non-local Cauchy problem
\begin{align}
\label{eqDf}
\mathfrak{D}^\Phi_t u = f(u), \quad u(0)=c \geq 0.
\end{align} 
\end{theorem}
\begin{proof}
We have that
\begin{align*}
\widetilde{u}(\lambda) = \frac{\Phi(\lambda)}{\lambda} \int_0^\infty v(x) e^{-x \Phi(\lambda)} dx = \frac{\Phi(\lambda)}{\lambda}\, \widetilde{v}(\Phi(\lambda)), \quad \Phi(\lambda)>w
\end{align*}
and the equation \eqref{eqDf} leads to
\begin{align*}
\Phi(\lambda)\widetilde{u}(\lambda) - \frac{\Phi(\lambda)}{\lambda} u(0) = f(\widetilde{u}(\lambda))
\end{align*}
From the linearity of $f$, we can write
\begin{align*}
\Phi(\lambda) \widetilde{v}(\Phi(\lambda)) - u(0) = f(\widetilde{v}(\Phi(\lambda)))
\end{align*}
Set $\lambda^* = \Phi(\lambda)$. From the fact that $u(0)= v(0)$ we write
\begin{align*}
\lambda^* \widetilde{v}(\lambda^*) - v(0) = f(\widetilde{v}(\lambda^*))
\end{align*}
which holds, for $f$ linear, if and only if $v^\prime = f(v)$.  
\end{proof}
%
%

Let us consider $\Phi(\lambda) = \lambda^\alpha$. The well-known case 
$$f(v)=-a v, \quad a >0$$ 
brings to the solution
\begin{align}
\label{M-L}
u(t) = E_\alpha(-a t^\alpha) = \sum_{k \geq 0} \frac{(-a t^\alpha)^k}{\Gamma(\alpha k + 1)} 
\end{align}
which is the Mittag-Leffler function. Thus, equation \eqref{eqv1} gives $v(x)= c e^{-a x}$ and equation \eqref{eqDv1} gives (as proved in \cite{Bingham1971})
\begin{align}
\label{M-L-u}
u(t) = \mathbf{E}_0[e^{-a L_t}].
\end{align}
We also notice that, from 
\begin{align*}
\sum_{k \geq 0} \frac{(-a t^\alpha)^k}{\Gamma(\alpha k + 1)} = \sum_{k \geq 0} \frac{(-a)^k}{k!} \mathbf{E}_0[(L_t)^k] = \mathbf{E}_0 \left[ e^{-a L_t}\right]
\end{align*}
we can write
\begin{align}
\label{momentLstable}
\mathbf{E}_0[(L_t)^k] = \frac{\Gamma(k+1)}{\Gamma(\alpha k + 1)} t^{\alpha k}, \quad k \in \mathbb{N}_0.
\end{align}

We focus on $\Phi$ given in \eqref{LevKinFormula} with 
\begin{align}
\label{assumptionPhi}
\mathtt{d}=0 \quad \textrm{and} \quad \mathtt{k}=0.
\end{align}
Further on we always assume that \eqref{assumptionPhi} applies. 

\section{Non-local non-linear equations}

Let us consider 
\begin{align*}
f(z)=z(1-z).
\end{align*}
We now approach the problem to find a probabilistic representation for the fractional logistic equation. In particular, we consider the following two cases involving inverses to subordinators.

\subsection{CASE I}
The solution $v$ to the logistic equation
\begin{align*}
\frac{dv}{dt} = f(v), \quad v(0)=v_0 \in (0, 1)
\end{align*}
on the positive real line can be written as
\begin{align}
\label{logv}
v(t) = \frac{v_0}{v_0 + (1-v_0) e^{-t}} = \sum_{k \geq 0} \left( \frac{v_0 -1}{v_0} \right)^k e^{-k t}, \quad t\geq 0.  
\end{align}
We denote by $\mathbf{V}ar[Y]$ the variance of $Y$, that is $\mathbf{V}ar[Y] = \mathbf{E}[(Y - \mathbf{E}[Y])^2]$.

Let us introduce the process $v(L_t)$ whose realizations include plateaux according to the random time $L_t$. We have that
\begin{align*}
v(L_t) = v_0 + \int_0^{L_t} f(v(s)) \, ds
\end{align*}
or equivalently
\begin{align*}
v(L_t) = v_0 + \int_0^\infty f(v(s)) \, \mathbf{1}_{(s < L_t)} ds
\end{align*}
where $L_t$ can be regarded as the first time the subordinator $H_s$ exits the set $(0,t)$, that is $(s<L_t) \equiv (t> H_s)$ under $\mathbf{P}_0$. Thus, 
\begin{align}
\label{formulavL}
\mathbf{E}_0[v(L_t)] = v_0 + \int_0^\infty f(v(s)) \, \mathbf{P}_0(H_s < t)\, ds.
\end{align}
The $\lambda$-potential
\begin{align}
\label{Lpot}
\mathbf{E}_0 \left[ \int_0^\infty e^{-\lambda t} \, v(L_t)\, dt \,;\, L_t < \zeta \right] = \frac{\Phi(\lambda)}{\lambda} \mathbf{E}_0 \left[ \int_0^\zeta e^{- t \Phi(\lambda)} v(t)\, dt \right], \quad \lambda>0
\end{align}
can be associated with \eqref{formulavL} only if $\zeta = \infty$ almost surely. 
Formula \eqref{Lpot} can be obtained from \eqref{potentialH} and \eqref{potentialL}, see \cite{CapDovDR} for details. Let us consider $\zeta$ such that $\zeta=T<\infty$ almost surely. From \eqref{Lpot} we observe that, as $\lambda\to 0^+$, we obtain for $u(t)= \mathbf{E}_0[v(L_t); L_t < T]$,
\begin{align}
\label{integrabilityu}
\int_0^\infty u(t)\, dt = \left( \lim_{\lambda \downarrow 0} \frac{\Phi(\lambda)}{\lambda} \right) \int_0^T v(t)\, dt
\end{align} 
which is finite only if the limit $\Phi(\lambda)/\lambda$ is finite. Since $v$ is continuous and bounded, formula \eqref{formulavL} is finite. In order to have a finite integral in \eqref{integrabilityu}, the function $u$ is obtained by extension with zero for $t\geq H_T$, that is for $L_t \geq T$. 

Assume that $v,u \in L^1((0,T^*))$ for some $T^*$. Then, formula \eqref{integrabilityu} can be considered in order to have a reading in terms of delayed and rushed growth (see \cite{CapDovDR}). An helpful example of this phenomenon is presented in Figure \ref{figDR}.

\begin{theorem}
\label{thm:var}
Let $\sigma \in C_b([0, \infty))$. Let $u \in \mathcal{M}^\prime_0$ be the solution to 
\begin{align*}
\mathfrak{D}^\Phi_t u + \sigma = f(u), \quad u(0)=u_0 \in (0,1).
\end{align*}
Then, $u(t)= \mathbf{E}_0[v(L_t)]$ if and only if $\sigma(t)=\mathbf{V}ar[v(L_t)]$.
\end{theorem}

\begin{proof}

With \eqref{eqL} in mind, an integration by parts yields 
\begin{align*}
\mathcal{D}^\Phi_t u(t) =  - \int_0^\infty v(x) \, \frac{dl}{dx}(t,x)\, dx =  v_0 \, l(t,0) + \mathbf{E}_0[v^\prime(L_t)]
\end{align*}
where
\begin{align*}
\lim_{x \downarrow 0} \int_0^\infty e^{-\lambda t} l(t,x)\, dt =\lim_{x \to 0} \frac{\Phi(\lambda)}{\lambda} e^{-x \Phi(\lambda)} = \frac{\Phi(\lambda)}{\lambda},
\end{align*}
that is,
\begin{align*}
\int_0^\infty e^{-\lambda t} l(t,0)\, dt = \int_0^\infty e^{-\lambda t} \Pi((t,\infty))\, dt.
\end{align*}
From \eqref{connection}, we have that
\begin{align*}
\mathfrak{D}^\Phi_t u(t) = \mathcal{D}^\Phi_t u(t)  - u(0)\, \Pi((t, \infty)) 
\end{align*}
and therefore, by taking into account that $u(0)=v_0$, we write
\begin{align*}
\mathfrak{D}^\Phi_t u(t) = \mathbf{E}_0[v^\prime(L_t)] = \mathbf{E}_0[f(v(L_t))].
\end{align*} 
Since $f(v)=v(1-v)$, we write
\begin{align*}
\mathbf{E}_0[v^\prime(L_t)] = \mathbf{E}_0[v(L_t) - v^2(L_t)] = u(t) - u^2(t) - \mathbf{E}_0[v^2(L_t) - u^2(t)].
\end{align*}
Now we observe that
\begin{align*}
\mathbf{E}_0[v^2(L_t) - u^2(t)]
= & \mathbf{E}_0[v^2(L_t)] - \big(\mathbf{E}_0[v(L_t)] \big)^2 = \mathbf{V}ar[v(L_t)]
\end{align*}

%
%
%
%
%
and this concludes the proof.
\end{proof}

Let us consider the case $\Phi(\lambda)=\lambda^\alpha$. With formula \eqref{logv} at hand, from \eqref{M-L} and \eqref{M-L-u}, we can write
\begin{align}
\label{funcWest}
u(t) = \sum_{k \geq 0} \left( \frac{u_0 -1}{u_0} \right)^k E_\alpha(-k t^\alpha), \quad t\geq 0.
\end{align}
This is the solution on the whole positive real line to 
\begin{align}
\label{eqModified}
D^\alpha_t u + \sigma = u(1-u), \quad u(0)=u_0 \in (0,1)
\end{align}
according to Theorem \ref{thm:var}. The function \eqref{funcWest} has been first considered in \cite{West} and after in \cite{ALN16} in order to have a good approximation of the solution to the fractional logistic equation (with Caputo-Djrbashian derivative). In \cite{Izadi2020} the author considered shifted-Legendre polynomials in order to obtain approximate solutions of the fractional-order logistic equation. In \cite{OB17} the authors considered a simple algorithm in order to obtain such solution including a numerical implementation in terms of Pad\`{e} approximation. Recently, it has been also considered in \cite{FCAAlog} as the solution to a modified fractional logistic equation which is related to \eqref{eqModified}. 

In conclusion, Theorem \ref{thm:var} extends the non-local logistic equation associated with \eqref{logv} to a general $\Phi$.

\subsection{CASE II}
The solution $v$ to the logistic equation
\begin{align*}
\frac{dv}{dt} = f(v), \quad v(0)=v_0 \in (0, 1)
\end{align*}
on the convergence set $(0, r)$ of the real line can be written as
\begin{align}
\label{vLog}
v(t) = \sum_{k \geq 0} E_k\, \frac{t^k}{k!}, \quad t \in (0, r)
\end{align}
where $E_0=v_0$ and the sequence $\{E_k\}_k$ is given by the Euler's numbers if $v_0 =1/2$. The equation
\begin{align*}
D^\alpha_t u = u(1-u)
\end{align*}
(where $D^\alpha_t$ is the Caputo-Djrbashian derivative) has been investigated in \cite{PhysA}. It turns out that, 
\begin{align}
\label{uLog}
u(t)= \sum_{k \geq 0} E_k^\alpha\, \frac{t^{\alpha k}}{\Gamma(\alpha k +1)}, \quad t \in (0, r_\alpha)
\end{align}
is written in terms of the coefficient $\{E_k^\alpha\}_k$ which are strictly related with the numbers $\{E_k\}_k$. In particular, we have that $E^1_k = E_k$, $\forall\, k \geq 0$ and, for the sequence $\{E^\alpha_k\}_k$, we have that 
\begin{align*}
& E^\alpha_0 = u_0 & \quad \textrm{(the initial datum)}\\
& E^\alpha_1 = E^\alpha_0 (1 - E^\alpha_0) & \quad \textrm{(the logistic constraint)}
\end{align*}
and 
\begin{align*}
E^\alpha_{k+1} = E^\alpha_k - \sum_{i=1}^k \left[\!\begin{matrix} \, k \, \\ \, i\,  \end{matrix}\!\right]_\alpha E^\alpha_i\, E^\alpha_{k-i}, \quad k \in \mathbb{N}
\end{align*}
is the generating recursive formula. We termed  
\begin{align}
\label{binomConvolutionML}
\left[\!\begin{matrix} \, k \, \\ \, i\,  \end{matrix}\!\right]_\alpha = \frac{\Gamma(\alpha k + 1)}{\Gamma(\alpha i+1) \Gamma(\beta (k-i)+1)}
\end{align}
as fractional binomial coefficient because of the many similar properties shared with the binomial coefficient. For example,
\begin{align*}
\left[\!\begin{matrix} \, k \, \\ \, i\,  \end{matrix}\!\right]_1 = \frac{k!}{i!\, (k-i)!} = \binom{k}{i}.
\end{align*} 
Straightforward calculations show that, for any $\sigma$, 
\begin{align*}
u(t) \neq \mathbf{E}_0[v(L_t)].
\end{align*}
Indeed, $u$ is defined on compacts $K \subseteq (0,r)$. However, we are still able to establish some connection between $u$ and the process $L_t$. Let us consider
\begin{align}
\label{phiMomL}
\phi_k(t) := \frac{1}{k!} \mathbf{E}_0[(L_t)^k], \quad k \in \mathbb{N}, \quad t \in (0, r)
\end{align}
which is the rescaled moment of order $k$ of $L_t$. For a suitable sequence $\{E^\Phi_k \}_k$, we introduce the function
\begin{align}
\label{functUbar}
\bar{u}(t) := \sum_{k \geq 0} E^\Phi_k \, \phi_k(t) = E^\Phi_0 \, \phi_0(t) + \sum_{k \geq 1} E^\Phi_k \, \phi_k(t)
\end{align}
where obviously $\phi_k(0)=0$ $\forall\, k>0$ and $\phi_0(t)=1$ $\forall\, t\geq 0$.

\begin{lemma}
\label{lemmaStable}
If $\Phi(\lambda) = \lambda^\alpha$ and $E^\Phi_k = E^\alpha_k$ for any $k$, then
\begin{align*}
u(t) = \bar{u}(t), \quad \forall\, t \in (0, r_\alpha).
\end{align*}
\end{lemma}
\begin{proof}
The proof follows immediately by comparing \eqref{uLog} with \eqref{functUbar}. Since $\mathfrak{D}^\Phi_t = D^\alpha_t$, the coefficients $E^\Phi_k$ are exactly given by $E^\alpha_k$ and, from \eqref{momentLstable},
\begin{align*}
\phi_k(t) = \frac{t^{\alpha k}}{\Gamma(\alpha k + 1)}.
\end{align*}
This concludes the proof.
\end{proof}

We notice that the representation \eqref{vLog} holds on compact sets. In particular, $v(t)$, $t<r$ implies that $v(L_t)$ is defined as $L_t < r$. That is, $t<H_r$. Thus, we should consider the probabilistic representation
\begin{align*}
\mathbf{E}_0[v(L_t); t < H_r]
\end{align*}
for the solution \eqref{uLog} where $H$ is a stable subordinator and $L$ is the inverse to $H$.

\begin{lemma}
\label{lemmaphi}
Let us consider \eqref{phiMomL}. We have that
\begin{align*}
\mathfrak{D}^\Phi_t \phi_k(t) = \phi_{k-1}(t), \quad t>0, \quad k \in \mathbb{N}.
\end{align*}
\end{lemma}
\begin{proof}
It follows immediately from the fact that (see formula \eqref{potentialL})
\begin{align*}
\int_0^\infty e^{-\lambda t} \int_0^\infty x^k\, l(t,x)\, dx\, dt = \frac{\Phi(\lambda)}{\lambda} \int_0^\infty x^k\, e^{-x \Phi(\lambda)}\, dx = \frac{\Phi(\lambda)}{\lambda} \frac{\Gamma(k+1)}{( \Phi(\lambda))^{k+1}}
\end{align*}
from which we get the Laplace transform 
\begin{align*}
\int_0^\infty e^{-\lambda t} \phi_k(t)\, dt = \frac{1}{\lambda\, (\Phi(\lambda))^k}.
\end{align*}
This formula has been obtained in \cite{VeilletteTaqqu2010}. From \eqref{PhiConv} and the fact that $\phi_k(0) = 0$, we write
\begin{align*}
\Phi(\lambda) \frac{1}{\lambda\, (\Phi(\lambda))^k} - \frac{\Phi(\lambda)}{\lambda} \phi_k(0) = \frac{1}{\lambda\, (\Phi(\lambda))^{k-1}}.
\end{align*}
We conclude the proof.
\end{proof}

%
%
%
%
%
%

We observe that for the sequence $E^\Phi_k = (-a)^k$, $k \geq 0$, $a>0$ we have that
\begin{align}
\label{unknownU}
\bar{u}(t) = \sum_{k \geq 0} \frac{(-a)^k}{k!} \mathbf{E}_0[(L_t)^k] = \mathbf{E}_0[e^{-a L_t}]
\end{align} 
where $L_t$ is an inverse to a subordinator with symbol $\Phi$. Such a representation holds for $t\geq 0$. Moreover, from Lemma \ref{lemmaphi} we are able to conclude that
\begin{align*}
\mathfrak{D}^\Phi_t \bar{u}(t) = \sum_{k \geq 0} (-a)^{k+1} \, \phi_k(t) = - a\, \bar{u}(t).
\end{align*}
This means that $\bar{u}(t) = \mathbf{E}_0[v(L_t)]$ where $v^\prime = -a v$. That is, the case in Theorem \ref{thm:linf}. 

Concerning the extension of the result in Lemma \ref{lemmaStable} to a general symbol $\Phi$, our conjecture is as follows: {\it There exists a sequence $\{E^\Phi_k\}_k$ such that 
\begin{align*}
\mathfrak{D}^\Phi_t \bar{u}(t) = f(\bar{u}(t)), \quad t \in (0, r_\Phi)
\end{align*}
for some $r_{\Phi}>0$. Moreover, 
\begin{align*}
\bar{u}(t) = \mathbf{E}_0[v(L_t); \, t < H_{r_\Phi}]
\end{align*}
where $H$ is a subordinator and $L$ is the inverse of $H$. 
}

The function \eqref{unknownU} has been also studied in \cite{Koc2011} for complete Bernstein functions and \cite{MT2019} for special Bernstein functions. It has been considered also in \cite{BucSak2019} in connection with the Poisson and Skellam processes. Moreover, the case $E^\Phi_k=a^k$ has been also studied in \cite{KocKon2019} for complete Bernstein functions and in \cite{Ascione} in the general case. As far as we know for the moments of $L_t$ and the function \eqref{unknownU} we do not have an explicit representation. Actually, this still is an open problem.

In conclusion, the above conjecture would extend the non-local logistic equation associated with \eqref{vLog} to a general $\Phi$.

\begin{figure}[tb]
\includegraphics[scale=.5]{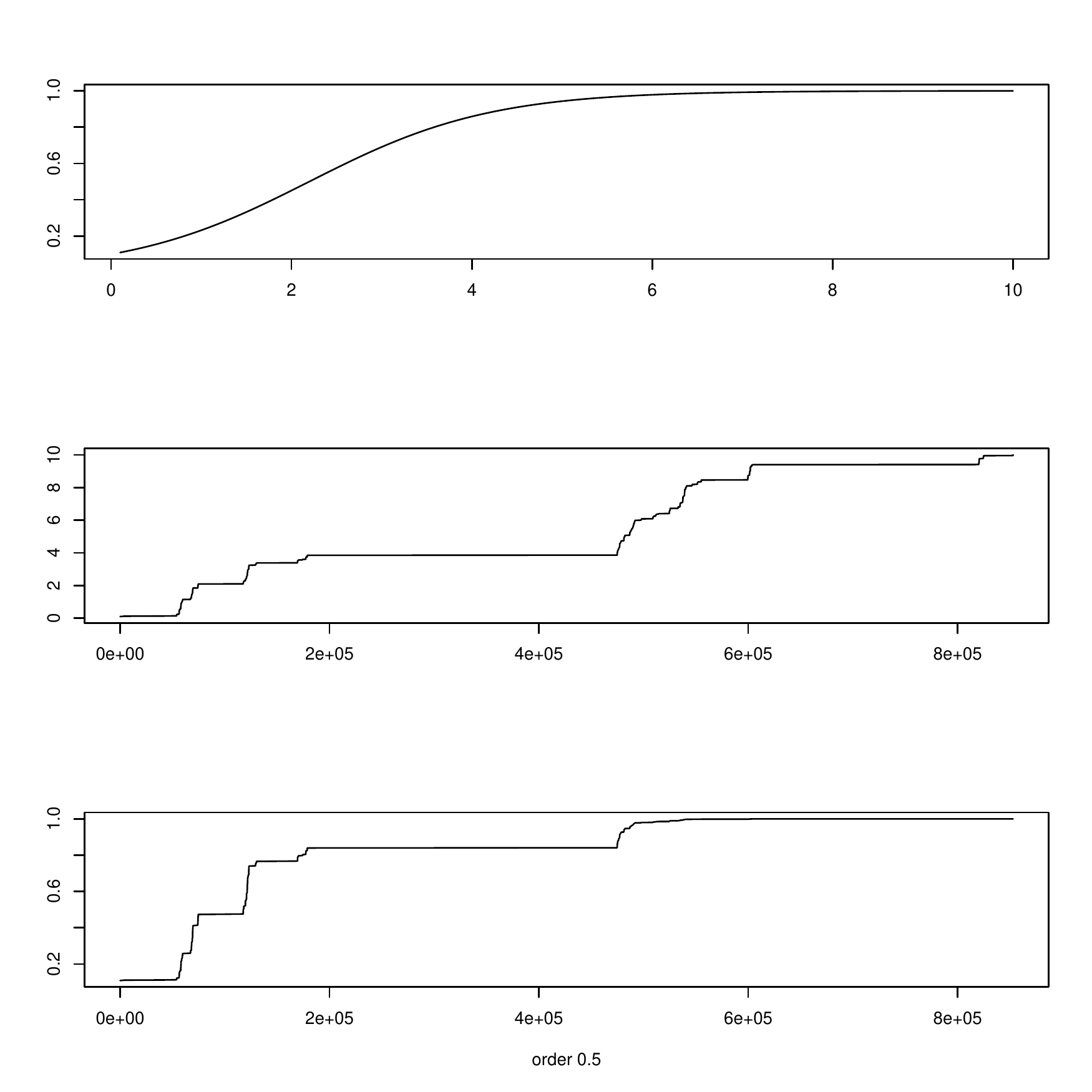}
\caption{Top picture) the function $(0,10) \ni t \to v(t) \in (0,1)$ given in formula \eqref{logv} with $v_0=0.1$; Picture in the middle) a realization of $L_t$, that is a continuous function from $(0,T)$ to $(0,10)$ where $T=8\mathrm{e}$+05. Here $L_t$ is the inverse to a stable subordinator with $\alpha=0.5$; Bottom picture) the composition $v(L_t) : (0,T) \to (0,1)$. The last picture shows that $L_t$ "delays" the profile of $v(t)$. Indeed, due to the plateaux of the new time $L_t$, from the function $v(t)$, $t \in (0, 10)$ we obtain the random function $v(L_t)$, $t \in (0, T)$ where $T>>10$. This can be regarded as a delaying effect of $L_t$ on the growth $v(t)$.}
\label{figDR}
\end{figure}



\begin{thebibliography}{10}

\bibitem{AN2021}
I. Area, J.J. Nieto,
\textit{Power series solution of the fractional logistic equation,}
Physica A: Statistical Mechanics and its Applications,
\textit{573}, 2021, 125947

\bibitem{ALN16} 
{I. Area, J. Losada, J. J. Nieto,}
\textit{A note on the fractional logistic equation,}
Physica A: Statistical Mechanics and its Applications, \textit{444} (2016) 182 - 187.

\bibitem{Ascione}
{G. Ascione.} 
\textit{Abstract Cauchy problems for generalized fractional calculus.} Nonlinear Analysis,  \textit{209}, August 2021, 112339

\bibitem{BM2001}
{B. Baeumer, M. M. Meerschaert,}  
\textit{Stochastic solutions for fractional Cauchy problems. }
Fractional Calculus and Applied Analysis, \textit{4}(4):481-500, 2001.

\bibitem{BerBook} 
{J. Bertoin,} 
Subordinators: Examples and Applications.
In: Bernard P. (eds) Lectures on Probability Theory and Statistics. Lecture Notes in Mathematics, vol 1717. Springer, Berlin, Heidelberg, 1999.

\bibitem{Bingham1971}
{N. H. Bingham.} 
\textit{Limit theorems for occupation times of markov processes.} Zeitschrift fur Wahrscheinlichkeitstheorie und verwandte Gebiete, \textit{17}(1):1-22, 1971.

\bibitem{BucSak2019}
{K. Buchak, L. Sakhno.} 
\textit{On the governing equations for Poisson and Skellam processes time-changed by inverse subordinators.} 
Theor. Probability and Math. Statist. 98 (2019), 91-104

\bibitem{CapDovDR} 
{R. Capitanelli, M. D'Ovidio,}
\textit{Delayed and Rushed motions through time change.}
ALEA, Lat. Am. J. Probab. Math. Stat. \textit{17} (2020), 183-204.

\bibitem{Chen2017}
{Z.-Q. Chen.} 
\textit{Time fractional equations and probabilistic representation.}
Chaos, Solitons \& Fractals, 102:168-174, 2017.

\bibitem{DCP}
{A. Di Crescenzo, P. Paraggio}
\textit{Logistic Growth Described by Birth-Death and Diffusion Processes,} 
Mathematics \textbf{7} (2019), Pag.1-28

\bibitem{DCDC2021}
A. Dom\'{e}nech-Carb\'{o}, C. Dom\'{e}nech-Casas\'{u}s,
\textit{The evolution of COVID-19: A discontinuous approach,}
Physica A: Statistical Mechanics and its Applications,
\textit{568}, 2021, 125752,

\bibitem{DOVspl2011}
{M. D'Ovidio}
\textit{On the fractional counterpart of the higher-order equations.}
Statistics and Probability Letters, \textit{81}, (2011), 1929 - 1939

\bibitem{PhysA} 
{M. D'Ovidio, P. Loreti,}
\textit{Solutions of fractional logistic equations by Euler's numbers.} 
Physica A: Statistical Mechanics and its Applications, \textbf{506} (2018) 1081 - 1092.

\bibitem{Izadi2020}
{M. Izadi.} 
\textit{A Comparative Study of Two Legendre-Collocation Schemes Applied to Fractional Logistic Equation.} Int. J. Appl. Comput. Math (2020) 6:71.

\bibitem{IzaSri2020}
M. Izadi, H.M. Srivastava, 
\textit{A Discretization Approach for the Nonlinear Fractional Logistic Equation.} 
Entropy 2020, 22, 1328.

\bibitem{KahPhaJam2020}
L.N. Kaharuddin, C. Phang, S.S Jamaian, 
\textit{Solution to the fractional logistic equation by modified Eulerian numbers.}
Eur. Phys. J. Plus \textit{135}, 229 (2020)

\bibitem{KumRaj2020}
A. Kumar,  Rajeev,
\textit{A moving boundary problem with space-fractional diffusion logistic population model and density-dependent dispersal rate,}
Applied Mathematical Modelling, \textit{88}, 2020, 951-965.

\bibitem{Koc2011}
{A. N. Kochubei.} 
\textit{General fractional calculus, evolution equations, and renewal processes.} 
Integral Equations and Operator Theory, \textit{71}(4):583-600, 2011.

\bibitem{KocKon2019}
{A. N. Kochubei and Y. Kondratiev.} 
\textit{Growth equation of the general fractional calculus.} Mathematics, \textit{7}(7):615, 2019.

\bibitem{FCAAlog} 
{M. D'Ovidio, P. Loreti, Sima Sarv Ahrabi,}
\textit{Modified Fractional Logistic Equation.}
Physica A: Statistical Mechanics and its Applications, \textit{505} (2018) 818 - 824.

\bibitem{MS2008}
{M. M. Meerschaert, H.-P. Scheffler.} 
\textit{Triangular array limits for continuous time random walks.} Stochastic processes and their applications, \textit{118}(9):1606-1633, 2008.

\bibitem{MS2013}
{M. M. Meerschaert, P. Straka.} 
\textit{Inverse stable subordinators.} 
Mathematical modelling of natural phenomena, \textit{8}(2):1-16, 2013.

\bibitem{MT2019}
{M. M. Meerschaert and B. Toaldo.} 
\textit{Relaxation patterns and semi-markov
dynamics.} 
Stochastic Processes and their Applications, \textit{129}(8):285-2879, 2019.

\bibitem{OB17} {M. Ortigueira, G. Bengochea,}
\textit{A new look at the fractionalization of the logistic equation}, Physica A: Statistical Mechanics and its Applications \textbf{467} (2017) 554-561.

\bibitem{SC2003}
{S.G. Samko, R.P. Cardoso,} 
\textrm{Integral equations of the first kind of Sonine
type.} Int. J. Math. Math. Sci. \textit{57}, 3609-3632 (2003)


\bibitem{SchilingBook}
{R. L. Schilling, R. Song, Z. Vondracek.} 
Bernstein functions: theory and applications, volume 37. Walter de Gruyter, 2012.

\bibitem{Toaldo2015}
{B. Toaldo.} 
\textit{Convolution-type derivatives, hitting-times of subordinators and time-changed $C_0$-semigroups.} 
Potential Analysis, \textit{42}(1):115-140, 2015.

\bibitem{VeilletteTaqqu2010}
{M. Veillette and M. S. Taqqu,} 
\textit{Using differential equations to obtain joint moments of first-passage times of increasing L\`{e}vy processes.} 
Statistics \& probability letters, \textit{80}(7-8):697-705, 2010.

\bibitem{West}
{B.J. West,}
\textit{Exact solution to fractional logistic equation}, 
Physica A: Statistical Mechanics and its Applications,  \textbf{429} (2015) 103-108.

\end{thebibliography}

\end{document}